\documentclass{article}
\usepackage{authblk}
\usepackage{abstract}
\usepackage{amsmath}
\usepackage{wasysym}
\usepackage{amsthm}
\usepackage{amssymb}
\usepackage[T1]{fontenc}

\newtheorem{Lemma}{Lemma}
\newtheorem{Theorem}{Theorem}
\newtheorem{Remark}{Remark}

\begin{document}
	\title{Skew generalized von Neumann-Jordan constant for the Bana\'s-Fr\k{a}czek space}
	\author[1]{Linhui Chen}
	\author[1]{Qi Liu\thanks{Qi Liu:liuq67@aqnu.edu.cn }}
	\author[1]{Xiewei Tan}
	\author[1]{Yuxin Wang}
	
	\affil{School of Mathematics and physics, Anqing Normal University, Anqing 246133,P.R.China}
	\maketitle 
	\begin{abstract}
		For any $\lambda>1$, $R_\lambda^2$ is Bana\'s-Fr\k{a}czek space, the exact value of the skew generalized von Neumann-Jordan constant $C^{p}_\mathrm{N J}(\xi,\eta,R_\lambda^2)$ is calculated. By careful calculations, $C_\mathrm{N J}^{p}(\xi,\eta,R_\lambda^2)=\frac{(\xi+\eta)^p+\left[(\eta+\xi)^2-\frac{4 \xi \eta}{\lambda^2}\right]^{p / 2}}{2^{p-1}\left(\xi^p+\eta^p\right)}$ is given.

	\end{abstract}
	\textbf{keywords:} {Banach spaces, geometric constants, Bana\'s-Fr\k{a}czek space}\\\textbf{Mathematics Subject Classification: }46B20
	
	\section{Introduction}
Throughout this article, let $X$ be a real Banach space with dimension at least 2. We will define $S_X$ as the unit sphere, that is $S_X=\left\{x \in X:\|x\|_X=1\right\}; B_X$ as the closed unit ball, that is $B_X=\left\{x \in X:\|x\|_X \leq 1\right\}$. Additionally, we will use $\operatorname{ex}\left(B_X\right)$ to indicate the collection of extreme points of $B_X$.
The geometric properties of Banach spaces are studied in terms of their geometric constants.

In retrospect,the von Neumann-Jordan(NJ) constant of a Banach space X was introduced by Clarkson \cite{01} as the minimum constant $C$ for which$$\frac{1}{C}\leq \frac{\vert\vert x+y\vert\vert^2+\vert\vert x-y\vert\vert^2}{2(\vert\vert x \vert\vert^2+\vert\vert y \vert\vert^2)} \leq C $$
holds for all $x,y$ $\in$ X with $(x,y)$ $\neq$ (0,0).

An equivalent definition of the von Neumann-Jordan(NJ) constant is found in \cite{02}:$$C_\mathrm{N J}(X)=\sup \bigg\{\frac{\vert\vert x+y \vert\vert^2+\vert\vert x-y \vert\vert^2}{2(\vert\vert x \vert\vert^2+\vert\vert y \vert\vert^2)}:x \in S_X,y \in B_X \bigg\}. $$

We know that\cite{04} the von Neumann-Jordan constant can be calculated by$$C_\mathrm{N J}(X)=\sup{\frac{\gamma_{X}(t)}{1+t^2}:0 \leq t \leq 1},$$ where $\gamma_{X}(t)=\sup \left\{\frac{ \vert\vert x+ty \vert\vert^2+\vert\vert x-ty \vert\vert^2}{2}:x,y \in S_X \right\}. $

The von Neumann-Jordan constant $C_{\mathrm{NJ}}(X)$ has been shown to be a useful tool for characterizing Hilbert spaces, uniformly non-square spaces, and super-reflexive spaces after a great deal of investigation by scholars. For additional information and findings about $C_{\mathrm{NJ}}(X)$, see \cite{05} \cite{06}.

While many scholars are studying the properties of the $C_{\mathrm{NJ}}(X)$ constant itself, other scholars have begun to generalize the constant. Thus, in \cite{03}, the generalized von Neumann-Jordan constant $C_\mathrm{N J}^{(p)}(X)$ is defined by$$C_{\mathrm{NJ}}^{(p)}(X)=\sup \left\{\frac{\|x+y\|^p+\|x-y\|^p}{2^{p-1}\left(\|x\|^p+\|y\|^p\right)}: x, y \in X,(x, y) \neq(0,0)\right\}.$$

Now let us collect some properties of this constant(see \cite{03}):\\
(i) $1 \leq C_\mathrm{N J}^{(p)}(X) \leq 2$;\\
(ii) X is uniformly non-square if and only if $C_\mathrm{N J}^{(p)}(X) <2$;\\
(iii) Let $L_r[0,1]$ denote the classical Lebesgue space $(1<r<\infty)$and let $1/r +1/r'=1,$ then $$C_\mathrm{N J}^{(p)}(L_{r}[0.1])=
\begin{cases}
	2^{2-p}, & 1<p\leq r;\\
	2^{p/r-p+1}, & r<p\leq r'.
\end{cases} $$

By using alternative geometric constants, we can better understand and analyze the characteristics of Banach spaces, which can give us information about the unique geometric characteristics of spaces. Therefore, it is beneficial for us to study the exact values of geometric constants in a specific space.
The authors presented the space $R_\lambda^2$ in \cite{08}, which is considered to be $R_\lambda^2:=\left(R^2,\|\cdot\|_\lambda\right)$, where $\lambda>1$ and

$$
\|(a, b)\|=\max \left\{\lambda|a|, \sqrt{a^2+b^2}\right\} .
$$

Naturally, a lot of academics are also fixated on figuring out how much this constant is in particular spaces. The authors then gave the precise value of $C_{\mathrm{NJ}}\left(R_\lambda^2\right)$ in \cite{07}, where $C_{\mathrm{NJ}}\left(R_\lambda^2\right)=2-1 / \lambda^2$.

Furthermore, the authors named the space as the Banaś-Frączek space.
The Banaś-Frączek space was generalized by C. Yang and X. Yang \cite{12} in 2016. $X_{\lambda, p}$ indicates this promotional space, which is thought to be where $\lambda>1$ and $p \geq 1, X_{\lambda, p}:=\left(R^2,\|\cdot\|_{\lambda, p}\right)$, of which $\|(a, b)\|_{\lambda, p}=\max \left\{\lambda|a|,\left(a^p+b^p\right)^{1/p}\right\}$, respectively.\\

The exact value of the $C_{\mathrm{NJ}}^{(p)}(X)$ constant for the Banaś-Frączek space, which was determined in 2018 by C. Yang and H. Li \cite{09} was presented by where $\lambda>1, p \geq 2$, for any $\lambda$ such that $\lambda^2\left(1-1 / \lambda^2\right)^{p/2} \geq 1, C_{\mathrm{NJ}}^{(p)}\left(R_\lambda^2\right)=1+\left(1-1 / \lambda^2\right)^{p/2}$. 

In a new research, Liu et al. \cite{10} introduced a new geometric constant with a skew connection that denotes to $L_{\mathrm{YJ}}(\xi, \eta, X)$ for $\xi, \eta>0$ as
$$
L_{\mathrm{YJ}}(\xi, \eta, X)=\sup \left\{\frac{\|\xi x+\eta y\|^2+\|\eta x-\xi y\|^2}{\left(\xi^2+\eta^2\right)\left(\|x\|^2+\|y\|^2\right)}: x, y \in X,(x, y) \neq(0,0)\right\}
$$
Additionally, the author\cite{11} provides a comparable constant
$$
L_{\mathrm{YJ}}^{\prime}(\xi, \eta, X)=\sup \left\{\frac{\|\xi x+\eta y\|^2+\|\eta x-\xi y\|^2}{2\left(\|x\|^2+\|y\|^2\right)}: x, y \in S_X\right\}
$$
In recent years, a few scholars began to conduct research on the $L_{\mathrm{YJ}}(\xi, \eta, X)$ constant. The authors \cite{14} compared the relationship between the $L_{\mathrm{YJ}}(\xi, \eta, X)$ constant and the $J_{\lambda, \mu}(X)$ constant, and established an inequality between these two constants. Furthermore, in the same year, the authors \cite{15} obtain the exact value of the $L_{\mathrm{YJ}}(\xi, \eta, X)$ constant and the $L_{\mathrm{YJ}}^{\prime}(\xi, \eta, X)$ constant for the regular octagon space.

The author introduce a skew generalized constant $C_\mathrm{N J}^p(\xi,\eta,X)$ \cite{13}based on the constant $L_\mathrm{YJ} (\xi, \eta, X)$ as follows: 
for $\xi,\eta>0$,
$$\begin{aligned}
	C_\mathrm{N J}^p(\xi,\eta,X)=\sup\left\{\frac{\|\xi x+\eta y\|^p+\|\eta x-\xi y\|^p}{2^{p-2}(\xi^p+\eta^p)(\|x\|^p+\|y\|^p)}:x,y\in X,(x,y)\neq(0,0)\right\}\\=\sup\left\{\frac{\|\xi x+\eta ty\|^{p}+\|\eta x-\xi ty\|^{p}}{2^{p-2}[(\xi^{p}+\eta^{p}t^{p})+(\eta^{p}+\xi^{p}t^{p})]}:x,y\in S_{X},0\leq t\leq1\right\},
\end{aligned}$$
where $1\leq p<\infty$.

The  Banaś-Frączek space have been studied widely in \cite{07,09}, some classical constants were calculated for these spaces. Now, in this article, we calculate and prove that for any $\lambda>1$ be such that $\lambda^2\left(1-1 / \lambda^2\right)^{p / 2} \geq 1$, $C_\mathrm{N J}^{p}(\xi,\eta,R_\lambda^2)=\frac{(\xi+\eta)^p+\left[(\eta+\xi)^2-\frac{4 \xi \eta}{\lambda^2}\right]^{p / 2}}{2^{p-1}\left(\xi^p+\eta^p\right)}$.

	\section{Main results}
	The following theorem is our main result.
	\begin{Theorem}\label{t1}
		Let $\lambda >1$ and $p \geq 2$ be such that $\lambda^2\left(1-1 / \lambda^2\right)^{p / 2} \geq 1$, then$$C_\mathrm{N J}^{p}(\xi,\eta,R_\lambda^2)=\frac{(\xi+\eta)^p+\left[(\eta+\xi)^2-\frac{4 \xi \eta}{\lambda^2}\right]^{p / 2}}{2^{p-1}\left(\xi^p+\eta^p\right)}.$$
	\end{Theorem}
	To complete the proof the theorem, we first give the following lemmas.
	
	\begin{Lemma}\label{l1}
		Let $0 \leq t\leq1, \lambda>1, p \geq 2$ and $0 \leq y_1 \leq 1 / \lambda$.\\(i) If $$\begin{aligned}
			  \varphi\left(y_1\right)=&\xi^2+\eta^2 t^2-\frac{2 \eta \xi t y_1-\eta t y_1}{\lambda}+\frac{\xi}{\lambda^2}+\frac{\xi y_1 \sqrt{1-\lambda^2}}{\lambda \sqrt{1-y_1^2}}\\&+\frac{\eta t y_1^2 \sqrt{1-\lambda^{-2}}}{\sqrt{1-y_1^2}} 
			 +2 \eta \xi t \sqrt{1-\lambda^{-2}} \sqrt{1-y_1^2},
		\end{aligned}$$then 
		\begin{equation*}
			\varphi\left(y_1\right) \leq(\xi+\eta t)^2+\frac{2[\xi-(2\eta \xi-\eta)t]}{\lambda^2}.
		\end{equation*} \\(ii)If $$f\left(y_1\right)=\left(\xi+\eta t \lambda y_1\right)^p+\left[\xi^2+\eta^2 t^2-\frac{2 \eta \xi t y_1}{\lambda}+2 \eta \xi t \sqrt{1-\lambda^{-2}} \sqrt{1-y_1^2}\right]^{p / 2},$$then 
		\begin{align*}
			f(y_{1})\leq&\max\bigg\{
				\left(\xi+\eta t\right)^{p}+\left(\xi+\eta t\right)^{p-2}\cdot\frac{2\left[\xi-(2\eta \xi-n)t\right]}{\lambda^{2}};\xi^{2}+\left(\xi+\eta t\right)^{p};\\& \left(\xi+\eta t\right)^p+\left(\xi^{2}+\eta^{2}t^2-\frac{4\eta \xi t}{\lambda^{2}}+2\eta \xi t\right)^{p / 2}\bigg\}.
		\end{align*} 
	\end{Lemma}

	\begin{proof}
		(i)since
		
		$$
		\varphi^{\prime}\left(y_1\right)=\frac{\eta t-2 \eta \xi t}{\lambda}+\frac{\xi \sqrt{1-\lambda^{-2}}}{\lambda} \cdot\left(1-y_1^2\right)^{-\frac{3}{2}}+\frac{\sqrt{\lambda^2-1}}{\lambda} n t y_1^3\left(1-y_1^2\right)^{-3 / 2}
		$$
			is an increasing function, we have $\varphi^{\prime \prime}(y_1) \geq 0$. Therefore,
		
		$$
		\varphi\left(y_1\right) \leq \max \{\varphi(0), \varphi(1 / \lambda) \}=\varphi(1 / \lambda)=(\xi+\eta t)^2+\frac{2[\xi-(2 \eta \xi-\eta) t]}{\lambda^2}.
		$$
		(ii)we may suppose $t \in(0,1]$. Let $y_0 \in(0,\frac{1}{\lambda})$ be such that $f^{\prime}\left(y_0\right)=0.$Then
		\begin{align*}
		&(\xi+\eta t \lambda y_0)^{p-1}\\&=\bigg[\lambda^{-2}+y_0 \sqrt{1-\lambda^{-2}} \cdot \lambda^{-1}\left(1-y_0^2\right)^{-\frac{1}{2}}\bigg]\left(\xi^2+\eta^2 t^2-\frac{2 \eta \xi t y_1}{\lambda}+2 \eta \xi t \sqrt{1-\lambda^{-2}} \sqrt{1-y_1^2}\right)^{p / 2-1}.
	\end{align*}
	Therefore	\begin{equation}\label{e1}
			\begin{aligned}f(y_{0})&=(\xi ^{2}+\eta ^{2}t^2-\frac{2\eta \xi ty_{0}}{\lambda}+2\eta \xi t\sqrt{1-\lambda^{-2}}\sqrt{1-y_{0}^{2}})^{p/2-1}\cdot\varphi(y_{0})\\&\leq(\xi +\eta t)^{p}+(\xi +\eta t)^{p-2}\cdot\frac{2[\xi -(2\eta \xi-\eta )t]}{\lambda^{2}}.\end{aligned}
		\end{equation}
		Now, (ii) is valid in view of $f(0) \leq \xi^2+(\xi+\eta t)^p, f(1 / \lambda)=(\xi+\eta t)^p+$ $\left(\xi^2+\eta^2 t^2-\frac{4 \eta \xi t}{\lambda^2}+2 \eta \xi t\right)^{p / 2}$ and \eqref{e1}.
	\end{proof}
	
	\begin{Lemma}\label{l2}
		Let p>2 and $\left(x_1, y_1\right) \in[0,1 / \lambda] \times[0,1 / \lambda]$ for some $\lambda>1$ such that $\lambda^2\left(1-1 / \lambda^2\right)^{p / 2} \geq 1$ If t $\in[0,1]$, then for
		\begin{equation}\begin{aligned}
				F\left(x_1, y_1\right)=&\lambda^p\left(\xi x_1+\eta t y_1\right)^p\\&+\left(\xi^2+\eta^2 t^2-2 \eta \xi t x_1 y_1+2 \eta \xi t \sqrt{1-x_1^2} \sqrt{1-y_1^2}\right)^{p / 2},
			\end{aligned}
	\end{equation}   
	we have \begin{equation}\label{e3}
	\begin{aligned}
		&\max\bigg\{F(x_1,y_1):0\leq x_{1},y_{1}\leq \frac{1}{\lambda}\bigg\}\\&\leq \max\bigg\{(1+\lambda^{-2p/(p-2)})(\xi+\eta t)^{p};\\&(\xi+\eta t)^{p}+(\xi^{2}+\eta^2 t^2-\frac{4\eta\xi t}{\lambda^{2}}+2\eta \xi t)^{p/2};\\
		&(\xi+\eta)^{p}+(\xi+\eta t)^{p-2}\cdot\frac{2[\xi-(2\eta\xi-\eta)t]}{\lambda^{2}}\bigg\}.
	\end{aligned}
		\end{equation}
	\end{Lemma}
	
	\begin{proof}Let $t \in(0,1]$, suppose  $\max \left\{F\left(x_1, y_1\right): 0 \leq x_1, y_1 \leq 1 / \lambda\right\}$ is attained at some $\{\left(x_1, y_1\right) \in(0,1 / \lambda) \times(0,1 / \lambda)\}$. 
		
		Then $F_{x_1}\left(x_1, y_1\right)=F_{y_1}\left(x_1, y_1\right)=0$. This implies
	\begin{equation}\label{e4}
			\begin{aligned}
				&  \lambda^p ( \xi x_1+\eta t y_1)^{p-1}-\eta t y_1\left[\xi^2+\eta^2 t^2-2 \eta \xi t x_1 y_1+2 \eta \xi t \sqrt{1-x_1^2} \sqrt{1-y_1^2}\right]^{p / 2-1} \\
				& =x_1 \eta t \frac{\sqrt{1-y_1^2}}{\sqrt{1-x_1^2}}\left(\xi^2+\eta^2 t^2-2 \eta \xi t x_1 y_1+2 \eta \xi t \sqrt{1-x_1^2} \sqrt{1-y_1^2}\right)^{p / 2-1} 
			\end{aligned}
	\end{equation}and
		\begin{equation}\label{e5}
		\begin{aligned}
			&\lambda^p\left(\xi x_1+\eta t y_1\right)^{p-1}-\xi x_1\left(\xi^2+\eta^2 t^2-2 \eta \xi t x_1 y_1+2 \eta \xi t \sqrt{1-x_1^2} \sqrt{1-y_1^2}\right)^{p / 2-1} \\
			& =y_1 \xi \frac{\sqrt{1-x_1^2}}{\sqrt{1-y_1^2}}\left(\xi^2+\eta^2 t^2-2 \eta \xi t x_1 y_1+2 \eta \xi t \sqrt{1-x_1^2} \sqrt{1-y_1^2}\right)^{p / 2-1}.
		\end{aligned}
		\end{equation}	
		Now, multiplying \eqref{e4} by \eqref{e5} , we get
		\begin{equation}
			\begin{aligned}
				 &\lambda^p\left(\xi x_1+\eta t y_1\right)^{p-2}\\&=\left(\xi^2+\eta^2 t^2-2 \eta \xi t x_1 y_1+2 \eta \xi t \sqrt{1-x_1^2} \sqrt{1-y_1^2}\right)^{p / 2-1}.
			\end{aligned}				
		\end{equation}
		Thus
		\begin{equation}
			\begin{aligned}
				&\lambda^p\left(\xi x_1+\eta t x_1\right)^p \\&\leq \lambda^{-2 p /\left(p-2\right)}\left(\xi^2+\eta^2 t^2-2 \eta \xi t x_1 y_1+2 \eta \xi t \sqrt{1-x_1^2} \sqrt{1-y_1^2}\right)^{p / 2},
			\end{aligned}
		\end{equation}
		which implies
		
		$$
		F\left(x_1, y_1\right) \leq\left(1+\lambda^{-2 p /\left(p-2\right)}\right)(\xi+\eta t)^{p}.
		$$

		To complete the proof of \eqref{e3}, we only need to prove \eqref{e8}-\eqref{e11} below.
		
		Firstly, we note that if $\lambda>1, p>2$ , $\lambda^2\left(1-1 / \lambda^2\right)^{p / 2} \geq 1$, then $\lambda \geq \sqrt{2}$ and
		\begin{equation}\label{e8}
			\begin{aligned}
				F\left(0, y_1\right) &\leq(\eta t)^p+(\xi+\eta t)^p \\&\leq \eta^p+(\xi+\eta t)^p \\&\leq(\xi+\eta t)^p+\left(\xi^2+\eta^2 t^2-\frac{4 \eta \xi t}{\lambda^2}+2 \eta \xi t\right)^{p / 2}.
			\end{aligned}
		\end{equation} 	
Secondly,we clearly have
		\begin{equation}
			\begin{aligned}
				F\left(x_1, 0\right)&=\lambda^p\left(\xi x_1\right)^p+\left[\xi^2+\eta^2 t^2+2 \eta \xi t \sqrt{1-x_1^2}\right]^{p / 2} \\&\leq \xi^p+(\xi+\eta t)^p\leq(\xi+\eta t)^p+\left(\xi^2+\eta^2 t^2-\frac{4 \xi \eta t}{\lambda^2}+2 \eta \xi t\right)^{p / 2}.
			\end{aligned}
		\end{equation}
Finally, from Lemma \ref{l2} and $\lambda \geq \sqrt{2}$, we also have
		\begin{equation}
				\begin{aligned}
				F \left(1/\lambda , y_1\right)&=\left(\xi+\eta t \lambda y_1\right)^p\\&+\left[\xi^2+\eta^2 t^2-\frac{2 \eta \xi t y_1}{\lambda}+2 \eta \xi t \sqrt{1-\lambda^{-2}} \sqrt{1-y_1^2}\right]^{p / 2} \\&
				\leq \max \left\{(\xi+\eta t)^p+\left(\xi^2+\eta^2 t^2-\frac{4 \eta \xi t}{\lambda^2}+2 \eta \xi t\right)^{p / 2} ;\right. \\&
				\left.(\xi+\eta t)^p+(\xi+\eta t)^{p-2} \cdot \frac{2 \cdot[\xi-(2 \eta \xi-\eta) t]}{\lambda^2}\right\}
			\end{aligned}
		\end{equation}
		and
	\small	\begin{equation}\label{e11}
		\begin{aligned}
		  F\left(x_1, 1/ \lambda\right)&=\left(\lambda \xi x_1+\eta t\right)^p\\&+\left(\xi^2+\eta^2 t^2-\frac{2 \eta \xi t x_1}{\lambda}+2 \eta \xi t\sqrt{1-x_1^2}\sqrt{1-\lambda^{-2}}\right)^{p / 2} \\
		& \leq\left(\xi+\eta t \lambda x_1\right)^p\\&+\left(\xi^2+\eta^2 t^2-\frac{2 \eta \xi t x_1}{\lambda}+2 \eta \xi t \sqrt{1-x_1^2} \sqrt{1-\lambda^{-2}}\right)^{p / 2} \\
			& \leq \max \left[(\xi+\eta t)^p+\left(\xi^2+\eta^2 t^2-\frac{4 \eta \xi t}{\lambda^2}+2 \eta \xi t\right)^{p / 2}\right. \text {; } \\
			& \left.(\xi+\eta t)^p+(\xi+\eta t)^{p-2} \cdot \frac{2-[\xi-(2\eta \xi-\eta) t]}{\lambda^2}\right\}.
		\end{aligned}
		\end{equation}
			\end{proof}
		\noindent{\bf	Proof of Theorem 1}
		Assume $\lambda>1$ is such that  $\lambda^2\left(1-1 / \lambda^2\right)^{p / 2} \geq 1$.
		Note that$$\mathrm{ex}(B_{X})=\{(z_1,z_2): z_1^2+z_2^2=1,|z_1|\leq1/\lambda\}.$$
		Now we prove that\begin{equation}
			C_\mathrm{N J}^{p}(\xi, \eta, R_\lambda^2) \leq \frac{(\xi+\eta)^p+\left[(\xi+\eta)^2-\frac{4 \xi \eta}{\lambda^2}\right]^{p / 2}}{2^{p-1}\left(\xi^p+\eta^p\right)} 
		\end{equation}
	 for all $x,y\in ex(B_X)$ is valid.
	 
	  Let $x=(x_1,x_1), y=(y_1,y_2)$. Now we consider the following  three cases:\\
		\textbf{Case(1)} If $$\vert\vert\xi x+\eta ty\vert\vert_{2}\leq\vert\lambda(\xi x_{1}+\eta ty_{1})\vert,\vert\vert\eta x-\xi ty\vert\vert_{2}\leq\vert\lambda(\eta x_{1}-\xi ty_1)\vert,$$ then
		$$\begin{aligned}
			& \vert\vert\xi x+\eta ty\vert\vert^{p}+\vert\vert\eta x-\xi ty_1\vert\vert^{p}=\lambda^{p}[\vert\xi x_{1}+\eta ty_{1}\vert^{p}+\vert\eta x_{1}-\xi ty_{1}\vert^{p}] \\
			& =2^{p-2}\lambda^{p}(\xi^{p}+\eta^{p})(x_{1}^{p}+t^{p}y_{1}^{p})\leq2^{p-2}(\xi^{p}+\eta^{p})(1+t^{p}).
		\end{aligned}$$
	\textbf{Case(2)} If$$\begin{aligned}
			& \|\xi x+\eta ty\|_{2}>|\lambda(\xi x_{1}+\eta ty_{1})|,\|\eta x-\xi ty\|_{2}>|\lambda(\eta x_{1}-\xi ty_{1})|,\end{aligned}$$ then$$\begin{aligned}
			\|\xi x+\eta ty||^{p}+||\eta x-\xi ty||^{p}&=\|\xi x+\eta ty\|_2^{p}+\|\eta x-\xi ty\|_{2}^{p} \\
			& \leq(\|\xi x+\eta ty\|_{2}^{2}+\|\eta x-\xi ty\|_{2}^{2})^{p/2}\\&=[(\xi^{2}+\eta^{2})(1+t^{2})]^{p/2}\\&\leq2^{p-2}(\xi^{p}+\eta^{p})(1+t^{p}).
		\end{aligned}$$
		\textbf{Case(3)} If
		$$\|\xi x+\eta ty\|_{2}\leq|\lambda(\xi x_{1}+\eta ty_{1})|~\text{and}~\|\eta x-\xi ty\|_{2}>|\lambda(\eta x_{1}-\xi ty_{1})|$$
		 or
	$$
			\|\xi x+\eta ty\|_{2}>|\lambda(\xi x_{1}+\eta ty_{1})|~\text{and}~\|\eta x-\xi ty\|_{2}\leq|\lambda(\eta x_{1}-\xi ty_{1})| 
		,$$
		then\begin{equation}
			\begin{aligned}
				& \vert\vert\xi x+\eta ty\vert\vert^{p}+\vert\vert\eta x-\xi ty\vert\vert^{p}\\&=\lambda^{p}|\xi x_{1}\pm \eta ty_{1}|^{p}+[(\xi x_{1}\mp \eta ty_{1})^{2}+(\xi x_{2}\mp \eta ty_{2})^{2}]^{p/2} \\
				& \leq\lambda^{p}|\xi x_{1}\pm \eta ty_{1}|^{p}+[\xi ^{2}+\eta ^{2}t^{2}\mp 2\eta \xi tx_{1}y_{1}+2\eta \xi t\sqrt{1-x_1^{2}}\sqrt{1-y_1^{2}}]^{p/2}.
			\end{aligned}
		\end{equation}
	Now,we may suppose that $x_1,y_1\in[0,1/\lambda]$. For
		\begin{align*}
			 F(x_1,y_{1})=\lambda^{p}|\xi x_{1}+\eta ty_{1}|^{p}+\bigg[\xi^{2}+\eta^{2}t^{2}-2\eta \xi tx_{1}y_{1}+2\eta \xi t\sqrt{1-x_{1}^{2}}\sqrt{1-y_{1}^{2}}\bigg]^{p/2},\end{align*} 
			We have \eqref{e1} by Lemma \ref{l2}.
		 
		 	Firstly, since $\lambda^{2 p /(p-2)} \geq \lambda^2$ and $\lambda^2\left(1-1 / \lambda^2\right)^{p / 2} \geq 1,$ we have
		$$\begin{aligned}
			\frac{\left(1+\lambda^{-2 p/(p-2)}\right)(\xi+\eta  t)^p}{\left.2^{p-2}\left(\xi^p+\eta^p\right) (1+t^p\right)}& \leq \frac{\left(1+\lambda^{-2}\right)(\xi+\eta  t)^p}{2^{p-2}\left(\xi^p+\eta^p\right)\left(1+t^p\right)} \\&\leq \frac{\left[1+\left(1-\frac{1}{\lambda^2}\right)^{p / 2}\right](\xi+\eta  t)^p}{2^{p-2}\left(\xi^p+\eta^p\right)\left(1+t^p\right)}\\&\leq \frac{(\xi+\eta t)^p+\left[\xi^2+\eta^2 t^2+2 \eta \xi t-\frac{4 \eta \xi t}{\lambda^2}\right]^{p / 2}}{2^{p-2}\left(\xi^p+\eta^p\right)\left(1+t^p\right) }.
		\end{aligned}$$
		
		Secondly,	since $t \in[0,1]$ and $\lambda^2\left(1-1 / \lambda^2\right)^{p / 2} \geq 1,$ we have $$\begin{aligned}
			 &\frac{(\xi+\eta t)^p+(\xi+\eta t)^{p-2} \cdot 2[\xi-(2 n \xi-\eta ) t] \lambda^{-2}}{2^{p-2}\left(\xi^p+\eta^p\right)\left(1+t^p\right)} \\&\leq \frac{(\xi+\eta t)^p+(\xi+\eta t)^p \cdot \lambda^{-2}}{2^{p-2}\left(\xi^p+\eta^p\right)\left(1+t^p\right)} \\		
			& \leq \frac{(\xi+\eta t)^p+\left[\xi^2+\eta^2 t^2+2 \eta \xi t-\frac{4 \eta \xi t}{\lambda^2}\right]^{p / 2}}{2^{p-2}\left(\xi^p+\eta^p\right)\left(1+t^p\right)}.
		\end{aligned}
		$$
		
		Thirdly, consider the following function:$$\Phi(t)=\frac{\left(\xi^2+\eta^2 t^2-\frac{4 \eta \xi t}{\lambda^2}+2 \eta \xi t\right)^{p / 2}}{2^{p-2}\left(\xi^p+\eta^p\right)\left(1+t^p\right)}.$$
		 It's easily to know for $\lambda \geq \sqrt{2}$, 
		$$\Phi(t)=\frac{\left(\xi^2+\eta^2 t^2-\frac{4 \eta \xi t}{\lambda^2}+2 \eta \xi t\right)^{p / 2}}{2^{p-2}\left(\xi^p+\eta^p\right)\left(1+t^p\right)}$$
		is increasing for $t \in[0,1]$. Hence we have
		$$
		\begin{aligned}
			  \frac{(\xi+\eta t)^p+\left(\xi^2+\eta^2 t^2-\frac{4 \eta \xi t}{\lambda^2}+2 \eta \xi t\right)^{p / 2}}{2^{p-2}\left(\xi^p+\eta^p\right)\left(1+t^p\right)} 
			\leq  \frac{(\xi+\eta)^p+\left[(\xi+\eta)^2-\frac{4 \xi \eta}{\lambda^2}\right]^{p / 2}}{2^{p-1}\left(\xi^p+\eta^p\right)}.
		\end{aligned}
		$$
	This implies that
		$$
 C_\mathrm{N J}^{p}(\xi, \eta, R_\lambda^2) \leq \frac{(\xi+\eta)^p+\left[(\xi+\eta)^2-\frac{4 \xi \eta}{\lambda^2}\right]^{p / 2}}{2^{p-1}\left(\xi^p+\eta^p\right)}.$$
 
 On the other hand, if we take     $x=\left(\frac{1}{\lambda}, \sqrt{1-\frac{1}{\lambda^2}}\right), \quad y=\left(\frac{1}{\lambda},-\sqrt{1-\frac{1}{\lambda^2}}\right)$, then we have
			$$\begin{aligned}
			C_\mathrm{N J}^{p}(\xi, \eta, R_\lambda^2) &\geq \frac{\left\|\xi x+\eta y\right\|^p+\left\|\eta x-\xi y\right\|^p }{2^{p-1}\left(\xi^p+\eta^p\right)}\\&=\frac{(\xi+\eta)^p+\left[(\xi+\eta)^2-\frac{4 \xi \eta}{\lambda ^2}\right]^{p/2}}{2^{p-1}\left(\xi^p+\eta^p\right)}.
			\end{aligned} $$
	    Hence we complete the proof of the above Theorem \ref{t1}.
	    
	   $C_\mathrm{NJ}(R^2_{\lambda})=2-1 / \lambda^2$ has given by Yang\cite{07} and $C_{\mathrm{NJ}}^{(p)}\left(R_\lambda^2\right)=1+\left(1-1 / \lambda^2\right)^{p/2}$ has also given by Yang \cite{09}. In fact, the exact value of the $C_\mathrm{NJ}(X)$ constant and the $C^{(p)}_\mathrm{NJ}(X)$ constant for the Bana\'s Fr\k{a}czek space can also estimated by Theorem\ref{t1}, just as the following remarks states.
	    \begin{Remark}\label{r1}
	    		For any $\lambda>1$, $R^2_{\lambda}$ is the Bana\'s Fr\k{a}czek space. Then$$C_\mathrm{NJ}(R^2_{\lambda})=C_\mathrm{N J}^{2}(1, 1, R_\lambda^2)=2-1 / \lambda^2.$$
	    \end{Remark}
	    \begin{Remark}
	    	For any $\lambda>1$, $R^2_{\lambda}$ is the Bana\'s Fr\k{a}czek space. Then$$C^{(p)}_\mathrm{NJ}(R^2_{\lambda})=C_\mathrm{N J}^{p}(1, 1, R_\lambda^2)=1+\left(1-1 / \lambda^2\right)^{p/2}.$$
	    \end{Remark}

							\end{document}